\documentclass[a4paper]{amsart}
\usepackage{amsfonts,amssymb,amscd,amsmath,enumerate,verbatim,newlfont,calc,}
\usepackage{amsmath}
\usepackage{amsmath, pb-diagram}
\usepackage{amsthm}
\usepackage{amsfonts}
\usepackage{amssymb}
\usepackage{latexsym}
\usepackage[all]{xy}

\theoremstyle{plain}

\newtheorem{thm}{Theorem}[section]

\newtheorem{lem}[thm]{Lemma}
\newtheorem{cor}[thm]{Corollary}

\theoremstyle{definition}
\newtheorem{defn}[thm]{Definition}
\newtheorem{exm}[thm]{Example}
\newtheorem{rem}[thm]{Remark}

\def\z*{\mbox{Z$^*$}}
\def\ZM*{\mbox{Z$_M^*$}}
\def\-z{\mbox{$\overline{Z}_M$}}
\def\s*{\mbox{(S$^*$)}}

\begin{document}

\title{Generalized injectivity and approximations}
\author{Serap \c{S}ahinkaya}
\address{Department of Mathematics, Gebze Institute of Technology, 41400 Gebze/Kocaeli,  Turkey}
\email{ssahinkaya@gyte.edu.tr}
\author{Jan Trlifaj}
\address{Department of Algebra, Charles University in Prague, Faculty of Mathematics and Physics Sokolovsk\' a 83, 186 75 Praha 8, Czech Republic}
\email{trlifaj@karlin.mff.cuni.cz}
\thanks{\c{S}ahinkaya was supported by grant SVV-2014-260107, and Trlifaj by GA\v CR 14-15479S}
\dedicatory{Dedicated to Alberto Facchini on his 60th birthday}
\begin{abstract} Injective, pure-injective and fp-injective modules are well known to provide for approximations in the category Mod-$R$ for an arbitrary ring $R$. We prove that this fails for many other generalizations of injectivity: the $C_1$, $C_2$, $C_3$, quasi-continuous, continuous, and quasi-injective modules. We show that, except for the class of all $C_1$-modules, each of the latter classes provides for approximations only when it coincides with the injectives (for quasi-injective modules, this forces $R$ to be a right noetherian V-ring, in the other cases, $R$ even has to be semisimple artinian). The class of all $C_1$-modules over a right noetherian ring $R$ is (pre)enveloping, iff $R$ is a certain right artinian ring of Loewy length $\leq 2$; in this case, however, $R$ may have an arbitrary representation type.
\end{abstract}
\date{\today}
\maketitle

\section{Introduction}

The importance of injective modules in algebra is based on the following two facts: their structure is known for many classes of rings, and each module has a unique injective envelope. Thus, minimal injective coresolutions exist and yield important homological invariants of modules, such as the Bass invariants \cite[\S 9.2]{EJ}.

It is easy to see that a module $I$ is injective, if and only if $I$ is {\em pure-injective} (i.e., each homomorphism $f : N \to I$ from a pure submodule $N$ of a module $M$ extends to $M$) and $I$ is {\em fp-injective} (i.e., $\hbox{Ext}^1_R(F,I) = 0$ for each finitely presented module $F$). These two more general notions of injectivity also fit well in approximation theory: pure-injective modules provide for envelopes (though they are not closed under extensions in general), and the fp-injective modules for special preenvelopes (though fp-injective envelopes need not exist in general, see \cite[14.62]{GT}).

There are other generalizations of injectivity; here, we will consider the ones studied in \cite[Chapter 2]{MM}:
\begin{defn} \label{MoMu} {\rm Let $R$ be a ring and $M$ a module. Then \\
$M$ is a {\em $C_1$-module} provided that every submodule of $M$ is essential in a direct summand of $M$;\\
$M$ is a {\em $C_2$-module} provided that $A$ is a direct summand in $M$ whenever $A$ is a submodule of $M$ such that $A$ isomorphic to a direct summand in $M$;\\
$M$ is a {\em $C_3$-module} in case the following holds true: if $A$ and $B$ are direct summands in $M$ and $A \cap B = 0$, then $A + B$ is a direct summand in $M$.}
\end{defn}

$C_1$-modules are also called {\em extending} or {\em $CS$-modules}. Clearly, each uniform module is $C_1$.

It is easy to see that each $C_2$-module is also a $C_3$-module. Conversely, for each module $M$, if $M \oplus M$ is a $C_3$-module, then $M$ is a $C_2$-module. However, $C_3$ is a weaker property in general: if $R$ is any integral domain which is not a field, then $R$ is $C_3$, but not $C_2$.

\begin{defn} \label{more} {\rm A module $M$ is \emph{quasi-injective} in case each homomorphism $g : N \to M$ from a submodule $N$ of $M$ extends to $M$. It is easy to see that $M$ is quasi-injective iff $M$ is a fully invariant submodule of its injective hull. For example, each semisimple module is quasi-injective.

A module $M$ is \emph{continuous}, if $M$ is both $C_1$ and $C_2$; $M$ is \emph{quasi-continuous} if $M$ is both $C_1$ and $C_3$.}
\end{defn}

The following chain of implications is well known and easy to prove for any module $M$: $M$ is injective $\Rightarrow$ $M$ is quasi-injective $\Rightarrow$ $M$ is continuous $\Rightarrow$ $M$ is quasi-continuous.

In order to simplify our notation, we let $\mathcal C _i$ denote the class of all $C_i$-modules for $i = 1,2,3$. Moreover, $\mathcal C _4$, $\mathcal C _5$, and $\mathcal C _6$ will denote the classes of all quasi-continuous, continuous, and quasi-injective modules, respectively. Thus, we have the following inclusions
$$(\clubsuit) \quad \mathcal C _2 \subseteq \mathcal C _3 \quad \hbox{and} \quad \mathcal C_6 \subseteq \mathcal C _5 = \mathcal C _{1} \cap \mathcal C _{2} \subseteq \mathcal C _4 = \mathcal C _{1} \cap \mathcal C _{3} \subseteq \mathcal C _3.$$

\medskip
It is natural to ask whether these classes $\mathcal C _i$ of generalized injective modules provide for envelopes or preenvelopes. Our main goal here is to show that in contrast with the classes of all pure-injective and fp-injective modules, the classes $\mathcal C _i$ rarely provide for approximations in Mod-$R$, and analyze these rare cases in detail.

\section{Preliminaries}

We start with recalling basics from the approximation theory of modules.

Let $\mathcal C$ be a class of (right $R$-) modules. A homomorphism $g: M \rightarrow E$ is a {\em $\mathcal C$-preenvelope} (or a {\em left $\mathcal C$-approximation}) of a module $M$, provided that $E \in \mathcal C$ and each diagram
$$\begin{diagram}
 \node{M}\arrow{s,t}{g'}\arrow{e,t}{g} \node{E}\arrow{sw,t,..}{\alpha}\\
 \node[1]{E ^\prime}
 \end{diagram}$$
with $E^\prime \in \mathcal C$ can be completed by $\alpha : E \to E^\prime$ to a commutative diagram. If moreover the diagram
$$\begin{diagram}
 \node{M}\arrow{s,t}{g}\arrow{e,t}{g} \node{E}\arrow{sw,t,..}{\alpha}\\
 \node[1]{E}
 \end{diagram}$$
can be completed only by an automorphism $\alpha$, we call $g$ a {\em $\mathcal C$-envelope} (or a {\em minimal left $\mathcal C$-approximation}) of $M$. It is easy to see that the $\mathcal C$-envelope is unique up to isomorphism. If each module has a $\mathcal C$-(pre)envelope, then $\mathcal C$ is called a \emph{(pre)enveloping} class of modules.

For example, if $\mathcal C$ is the class of all injective modules, then $\mathcal C$ is enveloping: a $\mathcal C$-envelope of a module $M$ is provided by the inclusion $M \hookrightarrow E(M)$ where $E(M)$ is the injective hull of $M$.

Dually, we define the notions of a {\em $\mathcal C$-precover} (= {\em right $\mathcal C$-approximation}) and a {\em $\mathcal C$-cover} (= a {\em minimal right $\mathcal C$-approximation}) of a module $M$, and of a \emph{(pre)covering} class of modules.

If $\mathcal C$ is the class of all injective modules, then $\mathcal C$ is (pre)covering, iff $R$ is a right noetherian ring (see e.g.\ \cite[5.4.1]{EJ}). For example, if $R$ is a Dedekind domain, then an injective cover of a module $M$ is easily seen to be provided by the embedding $D \hookrightarrow M$ where $D$ is the divisible part of $M$.

\medskip
It is easy to see that all the classes $\mathcal C _i$  ($1 \leq i \leq 6$) defined above are closed under isomorphisms and direct summands, so the following lemma applies to them:

\begin{lem}\label{easy}  Let $R$ be a ring and $\mathcal C$ be a preenveloping (precovering) class of modules closed under isomorphisms and direct summands. Then $\mathcal C$ is closed under direct products (direct sums).
\end{lem}
\begin{proof} Assume $\mathcal C$ is preenveloping and let $( E_i \mid i \in I )$ be a family of modules in $\mathcal C$. Let $f: P \to C$ be a $\mathcal C$-preenvelope of the module $P = \prod_{i \in I} E_i$. Denote by $\pi_i : P \to E_i$ the canonical projection ($i \in I$). Then there exist homomorphisms $g_i: C \to E_i$ such that $g_i f = \pi_i$ for each $i \in I$. Define a homomorphism $g : C \to P$ by $\pi_i g(c) = g_i(c)$ for all $c \in C$ and $i \in I$. Then $gf(x) = ( g_i(f(x)) \mid i \in I ) = x$ for all $x \in P$. Thus $P$ is isomorphic to a direct summand in $C$, and $P \in \mathcal C$ by our assumption on the class $\mathcal C$.

The proof for the precovering case is dual.
\end{proof}

\section{$C_i$-modules and approximations for $i > 1$}

For $i > 1$, the main obstacle for the existence of $\mathcal C _i$-preenvelopes comes the following simple lemma:

\begin{lem}\label{key_trick} Let $R$ be a ring and $2 \leq i \leq 6$. Let $N \in \mathcal C _i$ be a non-injective module. Then the module $M=N \oplus E(N)$ does not have a $\mathcal C _i$-preenvelope.
\end{lem}
\begin{proof} Assume that $f:M\rightarrow C$ is a $\mathcal C _i$-preenvelope of $M$. Since $E(M) \in \mathcal C _i$, we can assume that $f$ is monic, so w.l.o.g., $M\subseteq C$. Let $A = N \oplus 0 \subseteq M$ and $B = \{(n,n) \mid n \in N \} \subseteq M$. Then $A \cong N \cong B$ are direct summands in $M$ and $A \cap B = 0$, but $A + B$ is not a direct summand in $M$, because $N$ is not injective.

We claim that $A$ and $B$ are direct summands in $C$. Since $A$ is
a direct summand of $M$ we have a commutative diagram
 $$\begin{diagram}
 \node{A}\arrow{s,t}{id_A}\arrow{e,t}{\subseteq} \node{M}\arrow{sw,t}{\pi}\\
 \node[1]{A}
 \end{diagram}$$
Since $f:M\rightarrow C$ is a $\mathcal C_i$-preenvelope of $M$ and $A \in \mathcal C_i$, we have another commutative diagram
 $$\begin{diagram}
 \node{M}\arrow{s,t}{\pi}\arrow{e,t}{\subseteq} \node{C}\arrow{sw,t}{g_A}\\
 \node[1]{A}
 \end{diagram}$$
So the inclusion map $A \hookrightarrow C$ splits, and $A$ is a direct summand in $C$.

Similarly, $B$ is a direct summand in $C$. Since $C \in \mathcal C _i$, $C$ is a C3-module. Then $A + B$ is a direct summand in $C$, and hence in $M$, a contradiction.
\end{proof}

Now, we can prove that the classes $\mathcal C _i$ ($2 \leq i \leq 6$) provide for preenvelopes and precovers only if they coincide with the class of all injective modules:

\begin{thm}\label{rare} Let $R$ be a ring and $2 \leq i \leq 6$. Then the following conditions are equivalent:\\
(1) The class $\mathcal C _i$ is closed under finite direct sums; \\
(2) $\mathcal C _i$ coincides with the class of all injective modules; \\
(3) $\mathcal C _i$ is (pre)enveloping; \\
(4) $\mathcal C _i$ is (pre) covering. \\
If these conditions are satisfied, then $R$ is a right noetherian right V-ring; moreover, all semisimple modules are injective. Except for the case of $i = 6$, these conditions are further equivalent to \\
(5) $R$ is a semisimple artinian ring.
\end{thm}
\begin{proof} The implication (1) $\Rightarrow$ (2) follows by Lemma \ref{key_trick}. The implications (2) $\Rightarrow$ (3) and (5) $\Rightarrow$ (1) are clear, while (3) $\Rightarrow$ (1) and (4) $\Rightarrow$ (1) follow by Lemma \ref{easy}.

Next we prove that (2) implies (4). Since all semisimple modules are quasi-injective, (2) implies they are injective; in particular, $R$ is a right V-ring. We will prove that that $R$ is right noetherian. Assume this is not the case, so there is a strictly increasing chain of finitely generated right ideals in $R$
$$0=I_0\subsetneq I_1\subsetneq \cdots \subsetneq I_{2n}\subsetneq I_{2n+1}\subsetneq\cdots\subsetneq\cdots .$$
Let  $I=\bigcup_{n<\omega}I_{n}$. Since the module $I_{2n+1}/I_{2n}$ is non-zero and finitely generated, there is an epimorphism $\rho_n:  I_{2n+1}/I_{2n} \rightarrow S_n$ where $S_n$ is a simple module, for each $n < \omega$.

We will define $\varphi \in \hbox{Hom}_R(I,S)$, where $S=\bigoplus_{n<\omega}S_n$, as the union $\varphi = \bigcup_{n < \omega} \varphi_n$ where $\varphi_n \in \hbox{Hom}_R(I_{2n},\bigoplus_{m < n} S_m)$ and $\varphi_{n+1} \restriction I_{2n} = \varphi_n$ for each $n <\omega$.

First, $\varphi_0 = 0$. If $\varphi_n$ is defined, we use the injectivity of the semisimple module $T_n = \bigoplus_{m < n} S_m$ to extend $\varphi_n$ to $\eta_n \in \hbox{Hom}_R(I_{2n+1},T_n)$. Consider the canonical projection $\pi_n: I_{2n+1}\to I_{2n+1}/I_{2n}$. Then $0 \neq \rho_n \pi_n (x_n) \in S_n$, but $\rho_n \pi_n(I_{2n}) = 0$. We define $\varphi_{n+1}$ as an extension to $I_{2n + 2}$ of the morphism $\eta_n + \rho_n \pi_n : I_{2n+1} \to T_{n+1} = T_n \oplus S_n$. Notice that $\varphi_{n+1}(x_n) = \eta_n(x_n) + \rho_n \pi_n(x_n) \in T_{n+1} \setminus T_n$.

Since $S$ is semisimple, its injectivity yields an extension of $\varphi$ to some $\psi : R \to S$. The image of $\psi$ is contained in $T_n$ for some $n < \omega$, so $\varphi_{n+1}(x_n) = \varphi(x_n) = \psi(x_n) \in T_n$, a contradiction.

Since injective modules form a covering class over any right noetherian ring (see \cite[5.4.1]{EJ}), condition (4) holds.

In view of $(\clubsuit)$, it remains only to show (2) $\Rightarrow$ (5) for $i = 5$ (i.e., for the smallest class $\mathcal C _5$). But this has already been proved in \cite[Corollary 2]{HR}.
\end{proof}

In the case of quasi-injective modules (i.e., for $i = 6$), it is well known that each module has a (unique) minimal quasi-injective extension, namely the sum of all images of $M$ taken over all the endomorphisms of $E(M)$, see \cite[Corollary 1.15]{MM}. This extension, however, is not a quasi-injective (pre)envelope in general (just note that for each non-injective module $N$, the module $M = N \oplus E(N)$ has no quasi-injective preenvelope by Lemma \ref{key_trick}). In the quasi-injective case, we cannot say much more about the properties of the rings $R$ satisfying the equivalent conditions of Theorem \ref{rare}. In fact, we have

\begin{exm} \label{cozz} Let $R$ be any hereditary two-sided noetherian right V-ring. Then the classes of all quasi-injective and all injective modules coincide by \cite[Proposition 5.19(3)]{CF}. Hence the equivalent conditions of Theorem \ref{rare} are satisfied. (We refer to \cite[Chapter 5]{CF} for interesting constructions of such rings employing differential polynomials over universal differential fields.)
\end{exm}

\medskip
It is well-known that right noetherian rings are characterized by the property that the class of all injective modules is closed under arbitrary direct sums, and right hereditary rings by the injective modules being closed under homomorphic images. It turns out that we can relax this characterization by employing any of the larger classes $\mathcal C _i$ ($1 < i \leq 6$) studied here; indeed, the largest class $\mathcal C _3$ is sufficient:

\begin{lem}\label{Noetherian} The following conditions are equivalent for a ring $R$:\\
$(1)$ $R$ is right noetherian;\\
$(2)$ The direct sum of any set of injective modules is a $C_3$-module.
\end{lem}
\begin{proof} Clearly (1) $\Rightarrow$ (2).\\
(2)$\Rightarrow$ (1): Let $C=\bigoplus_{i\in I}E_i$ be a direct sum of injective modules. Then $M = C \oplus E(C)$ is also a direct sum of injective modules, so both $C$ and $M$ are $C_3$ by the assumption (2). By Lemma \ref{key_trick}, this implies that $C$ is injective, and proves (1).
\end{proof}

\begin{lem} The following conditions are equivalent for a ring $R$:\\
$(1)$ $R$ is right hereditary;\\
$(2)$ Every quotient module of an injective module is $C_3$.
\end{lem}
\begin{proof} Clearly $(1)\Rightarrow (2)$. \\
$(2)\Rightarrow (1)$ Let $I$ be an injective module and $M$ be a submodule of $I$. Then $I/M \oplus E(I/M)$ is a homomorphic image of the injective module $I \oplus E(I/M)$, hence it is $C_3$ by assumption. Since $I/M$ is $C_3$, too, Lemma \ref{key_trick} implies that $I/M$ is injective.
\end{proof}

\section{The case of $C_1$-modules}

In this section, we will deal with the remaining case of $i = 1$, that is, of $C_1$-modules. These modules are important in the decomposition theory of modules, see e.g.\ \cite{GG1} and \cite{GG2}. We start with a lemma showing that if $\mathcal C _1$ is closed under finite direct sums, then all {\lq}singular{\rq} $C_1$-modules are injective:

\begin{lem}\label{C1_inj} Let $R$ be a ring such that the class $\mathcal C _1$ is closed under finite direct sums.
Let $N$ be a $C_1$-module such that $\hbox{Hom}_R(N,E(R)) = 0$. Then $N$ is injective.
\end{lem}
\begin{proof} Assume that $N$ is not injective. By Baer's Criterion, $\hbox{Ext}^1_R(R/I,N) \neq 0$ for an essential right ideal $I$ of $R$.
So there exists an $f \in \mbox{Hom}_R(I,N)$ which does not extend to $R$.

Consider the module $M = E(R) \oplus N$ and its submodule $J = \{ (i,f(i)) \mid i \in I \}$. Clearly, $I \cong J$ via the map $e : i \mapsto (i,f(i))$.
Moreover, $M \in \mathcal C_1$ by our assumption on the class $\mathcal C _1$. So there are an essential extension $J \trianglelefteq U \subseteq M$ and a submodule $V \subseteq M$ such that $M = U \oplus V$. Then $E(R) \cong E(U)$, so $\hbox{Hom}_R(N,U) = 0$ by our assumption on $N$. It follows that $N$ is a direct summand in $V$, whence $U$ is isomorphic to a direct summand in $E(R)$, and $U$ is injective. Hence $e$ extends to an $h : R \to U$. Let $\pi : M \to N$ denote the canonical projection, and put $g = \pi h : R \to N$. Then $g \restriction I = f$, a contradiction.
\end{proof}

\begin{rem}\label{si} If $R$ is a right non-singular ring such that the class $\mathcal C _1$ is closed under finite direct sums, then Lemma \ref{C1_inj} yields that each singular module is completely reducible and injective, that is, $R$ is a right \emph{SI-ring} in the sense of \cite[Chap. III]{G}. In particular, $R$ is right hereditary.
\end{rem}

\begin{exm} \label{UT2} Let $R = UT_2(K)$ denote the upper-triangular $2 \times 2$ matrix ring over a skew-field $K$. Up to isomorphism, there are just three indecomposable modules: the simple projective $A$, the simple injective $B$, and the projective and injective $C$, and each module is a direct sum of copies of these modules. The singular modules are those isomorphic to direct sums of copies of $B$, so $R$ is a (hereditary) SI-ring. In fact, in this case $\mathcal C _1 = \hbox{Mod-}R$ (see Theorem \ref{mainC1} below).
\end{exm}

The following property of $C_1$-modules plays a key role in the noetherian setting:

\begin{lem}\label{unif} \cite[2.19]{MM} Let $R$ be a right noetherian ring. Then each $C_1$-module is a direct sum of uniform modules.
\end{lem}

\begin{thm} \label{mainC1} Let $R$ be a right noetherian ring. Then the following are equivalent:\\
$(1)$ The class $\mathcal C _1$ is (pre)enveloping.\\
$(2)$ $R$ is a right artinian ring such that each uniform module is either simple or injective of composition length $2$. \\
In this case $\mathcal C _1$ is the class of all modules of the form $S \oplus I$ where $S$ is semisimple and $I$ is injective.
\end{thm}
\begin{proof} Assume (1). By Lemma \ref{easy}, $\mathcal C _1$ is closed under direct products.

Let $M$ be a $C_1$-module. Then for each cardinal $\kappa$, also $M^\kappa \in \mathcal C_1$. By Lemma \ref{unif}, $M^\kappa \cong \bigoplus_i U_i$ where each $U_i$ is uniform, hence a submodule of an injective hull of a cyclic module. Let $\lambda _R$ denote the supremum of cardinalities of injective hulls of all cyclic $R$-modules. Then for each cardinal $\kappa$, $M^\kappa$ is a direct sum of $\leq \lambda _R$-generated modules. By \cite[Theorem 10]{H}, this implies that $M$ is a $\Sigma$-pure-injective module.

Moreover, each direct sum $D = \bigoplus_k C_k$ of $C_1$-modules is a pure submodule in the direct product $P = \prod_k C_k$. By the above, $P$ is $\Sigma$-pure-injective, so all its pure submodules are direct summands. In particular, $D$ is a $C_1$-module, so the class $\mathcal C _1$ is also closed under direct sums. By \cite[Corollary 3]{HJL}, $R$ is a right artinian ring and each uniform module $U$ has composition length at most $2$, whence
$U$ is either simple or injective.

Assume (2). By \cite[Lemma 5]{DS}, any direct sum of simple modules and injective modules of length $2$ is $C_1$, so $\mathcal C _1$ coincides with the class of all modules of the form $S \oplus I$ where $S$ is semisimple and $I$ is injective by Lemma \ref{unif}. Since $R/J$ is semisimple artinian, and semisimple modules are exactly the ones annihilated by the Jacobson radical $J$ of $R$, the class $\mathcal C _1$ is closed under direct products and, again by \cite[Theorem 10]{H}, all semisimple modules are $\sum$-pure-injective, and so are all modules in $\mathcal C _1$.

Let $\mathcal S \subseteq \mathcal U$ be a representative set of all simple modules and all indecomposable injectives. If $N$ is an arbitrary module, then there is only a set of homomorphisms from $N$ to the modules in $\mathcal S$. Let $u : N \to C$ be the product of these morphisms. Then $C \in \mathcal C _1$ by the above. If $M \in \mathcal C _1$, then $M$ is isomorphic to a direct sum of elements of $\mathcal S$, and hence to a direct summand in the direct product $P$ of those elements. Now, each homomorphism $f : N \to P$ factorizes through $u$. Since there is a split monomorphism $M \hookrightarrow P$, we infer that $u$ is a $\mathcal C_1$-preenvelope of $N$. Finally, all modules in $\mathcal C_1$ are pure-injective, so $N$ has a $\mathcal C _1$-envelope by \cite[5.11]{GT}, and (1) holds.
\end{proof}

\begin{cor} \label{comC1} Let $R$ be a commutative noetherian ring. Then the following are equivalent:\\
$(1)$ The class $\mathcal C _1$ is (pre)enveloping.\\
$(2)$ $R$ decomposes into a finite ring direct product $R = \prod_{i < m} R_i$, where each $R_i$ is an artinian local principal ideal ring of length $\leq 2$. \\
$(3)$ $\mathcal C _1 = \hbox{Mod-}R$.
\end{cor}
\begin{proof} (1) $\Rightarrow$ (2): Since each commutative artinian ring is a finite direct product of local rings \cite[p.312]{AF}, in view of Theorem \ref{mainC1}, we can assume that $R$ is a local artinian ring such that each uniform module is either simple, or injective of composition length $2$. Let $S$ denote the simple module. Then $\mbox{Soc}(R) = \bigoplus_{j < n} I_j$ is the unique maximal ideal of $R$, and $I_j \cong S$ for each $j < n$. If $n > 1$, then the locality of $R$ implies that $R/\bigoplus_{0 < j < n} I_j$ and  $R/\bigoplus_{j < n - 1} I_j$ are uniform modules of length $2$, so they are both isomorphic to $E(S)$. Since $R$ is commutative, necessarily $\bigoplus_{0 < j < n} I_j = \bigoplus_{j < n - 1} I_j$, a contradiction. This proves that $n = 1$, that is, $R$ is a principal ideal ring of length $\leq 2$.

(2) $\Rightarrow$ (3) follows by Theorem \ref{mainC1}, since each $R_i$ is of finite representation type, and if $R_i$ is not a field, then $R_i$ has just two representatives of indecomposable $R_i$-modules: the injective one, $R_i$, and the simple one, $R_i/\mbox{Soc}(R_i)$. The implication (3) $\Rightarrow$ (1) is obvious.
\end{proof}

By Theorem \ref{comC1}, the commutative rings $R$ such that $\mathcal C _1$ is preenveloping are necessarily of finite representation type. This is not the case in general. Our final example shows that there is no bound on the representation type even for hereditary artin algebras $R$ such that $\mathcal C _1$ is preenveloping: $R$ can be of finite, tame or wild type:

\begin{exm} \label{UT2KL} Let $K \subseteq F$ be skew-fields and $R = UT_2(F,K)$ the subring of $M_2(F)$ consisting of the matrices
$\begin{pmatrix}
f_1&f_2\\0&k
\end{pmatrix}$
with $f_1, f_2 \in F$ and $k \in K$ (so $R = UT_2(K)$ is the ring from Example \ref{UT2} in the particular case of $K = F$). Assume that $d = \dim F_K$ is finite. Then $R$ is right artinian and left and right hereditary (and $R$ is left artinian, iff $\dim {_K}F < \infty$). By the right-hand version of \cite[III.2.1]{ARS}, the category $\mbox{Mod-}R$ is equivalent to the category $\mathcal C$ of triples: right $R$-modules correspond to the triples $(A,B,f) \in \mathcal C$ such that $A$ is a right $F$-module, $B$ a right $K$-module, and $f \in \mbox{Hom}_K(A,B)$, while right $R$-homomorphisms correspond to the maps between triples $(A,B,f) \in \mathcal C$ and $(A^\prime,B^\prime,f^\prime) \in \mathcal C$ defined as the pairs $(\alpha, \beta)$ such that $\alpha \in \mbox{Hom}_F(A,A^\prime)$, $\beta \in \mbox{Hom}_K(B,B^\prime)$, and $\beta f = f^\prime \alpha$.

In this correspondence, indecomposable injective modules correspond to the triples $(F,0,0)$ and $(H,K,g)$, where $H$ is the right $F$-module consisting of all right $K$-homomorphisms from $F$ to $K$ (so $H \cong F$, because $d < \infty$), and $g : F \otimes_F H \to K$ is the right $K$-homomorphism defined by $g(f \otimes h) = h(f)$, see \cite[II.2.5.(c)]{ARS}. While the module corresponding to $(F,0,0)$ is simple, we have the exact sequence $0 \to (0,K,0) \to  (H,K,g) \to (H,0,0) \to 0$  (cf.\ \cite[II.2.3]{ARS}), which shows that the injective module corresponding to its middle term has length 2 (it has a simple socle corresponding to $(0,K,0)$, and a simple top corresponding to  $(F,0,0)$). By Theorem \ref{mainC1}, the class of all $C_1$-modules is preenveloping. However, if $d > 1$, then $\mathcal C _1 \neq \hbox{Mod-}R$ (as $P \notin \mathcal C _1$, where $P$ is the indecomposable projective module which is not simple).

Finally, by \cite[Theorem on pp.2-3]{DR}, the ring $R$ is of finite representation type for $d \leq 3$, it is of tame type for $d = 4$, and it is wild for $d \geq 5$.
\end{exm}


\medskip
\begin{thebibliography}{ARS}
\bibitem{AF}   Anderson, F. W.,  Fuller, K. R.: {\it Rings and Categories of
Modules}, 2nd ed., Springer-Verlag, New York 1992.
\bibitem{ARS}   Auslander, M., Reiten, I., Smalo, S.: {\it Representation Theory of Artin Algebras}, CSAM 36, Cambridge Univ. Press, Cambridge 1995.
\bibitem{CF}  Cozzens, J.,  Faith, C.: {\it Simple Noetherian Rings}, Cambridge Univ. Press, Cambridge 1975.
\bibitem{DR}  Dlab, V., Ringel, C. M.: {\it Indecomposable representations of graphs and algebras}, Mem. Amer. Math. Soc. 173, Providence 1976.
\bibitem{DS}  Dung, N. V.,  Smith, P. F.: {\it Rings for which certain modules are CS}, J. Pure Appl. Algebra \textbf{102}(1995), 273-287.
\bibitem{EJ}  Enochs, E.E.,  Jenda, O.M.G.: {\it Relative Homological Algebra}, 2nd revised and extended ed., Vol.1,
GEM \textbf{30}, W. de Gruyter, Berlin 2011.
\bibitem{GG1} Gomez Pardo, J. L., Guil Asensio, P. A.: {\it Every $\sum$-CS-module has an indecomposable decomposition}, Proc. Amer. Math. Soc. \textbf{129}(2001), 947-954.
\bibitem{GG2} Gomez Pardo, J. L., Guil Asensio, P. A.: {\it Indecomposable decompositions of modules whose direct sums are CS}, J. Algebra \textbf{262}(2003), 194-200.
\bibitem{G}  Goodearl, K.R. {\it Singular torsion and the splitting properties}, Mem. Amer. Math. Soc. 124, Providence 1972.
\bibitem{GT}  G\"obel, R.,  Trlifaj, J.: {\it Approximations and Endomorphism Algebras of Modules}, 2nd revised and extended ed.,
GEM \textbf{41}, W. de Gruyter, Berlin 2012.
\bibitem{H}  Huisgen-Zimmermann, B.: {\it Purity, algebraic compactness, direct sum decompositions, and representation type}, in Infinite Length Modules, Trends in Math., Birkh\" auser, Basel 2000, 331-367.
\bibitem{HJL} Huynh, D.V., Jain, S.K and Lopez Permouth, S.R.:  {\it Rings characterized by direct sums of CS modules}, Communications in Algebra, 28, \textbf{(9)}, (2000), 4219-4222.
\bibitem{HR} Huynh, D.V. and Rizvi, T.S.:  {\it An approach to Boyle's conjecture}, Proc. Edin. Math. Soc. \textbf{(40)}(1997), 267-273.
\bibitem{MM} Mohamed S.H., M\" uller B.J.: {\it Continous and Discrete Modules}, London Math. Soc. LNS \textbf{147}, Cambridge Univ. Press., Cambridge 1990.
\end{thebibliography}
\end{document}